\documentclass[11pt, a4paper]{amsart}
\usepackage{a4,amsmath,amssymb,epic}
\usepackage{graphicx}
\usepackage{multicol}
\usepackage{rotating}
\usepackage{lscape}
\usepackage{amssymb,amsthm}
\usepackage{enumitem}
\usepackage[all]{xy}
\usepackage{tikz}
\usepackage{young}
\usepackage[vcentermath,enableskew]{youngtab}

\allowdisplaybreaks
\numberwithin{equation}{subsection}
\theoremstyle{plain}
\newtheorem{lem}{Lemma}[subsection]
\newtheorem{prop}[lem]{Proposition}

\newtheorem{thm}[lem]{Theorem}

\newtheorem*{lem*}{Lemma}
\newtheorem*{Acknowledgements*}{Acknowledgements}
\newtheorem*{prop*}{Proposition}
\newtheorem*{thm*}{Theorem}
\newtheorem*{cor*}{Corollary}
\newtheorem*{conj*}{Conjecture}
\theoremstyle{remark}

\newtheorem{rmk}[lem]{Remark}
\newtheorem{defn}[lem]{Definition}

\newtheorem{eg}[lem]{Example}

\newtheorem{rem}[lem]{Remark}

\newcommand{\Hom}{\operatorname{Hom}}

\newcommand{\cha}{\operatorname{cha}}

\newcommand{\B}{{B_{m,p,n}	}}
\newcommand{\Bone}{{B_{m,1,n}}}

\newcommand{\ZZ}{{\mathbb Z}}
\newcommand{\NN}{{\mathbb N}}

\newcommand{\hcf}{\operatorname{hcf}}

\newcommand{\Stab}{\operatorname{Stab}}

\newcommand{\ft}{{\mathfrak{t}}}

\begin{document}
\title[Brauer
  algebras of type $G${\protect\lowercase{$(m,p,n)$}}]{Decomposition
  numbers for Brauer  algebras of type
  $G${\protect\lowercase{$(m,p,n)$}} in characteristic zero}  
\author{C.~Bowman}
\email{Bowman@math.jussieu.fr}
\address{Institut de Math\'ematiques de Jussieu, 175 rue du chevaleret,
  75013, Paris} 
\author{A.~G.~Cox} 
\email{A.G.Cox@city.ac.uk} 
%\author{M.~De Visscher}
%\email{Maud.Devisscher.1@city.ac.uk }
\address{Centre for Mathematical Science,
 City University London,
 Northampton Square,
 London,
 EC1V 0HB,
 England.}
 \date{\today}

 \maketitle
 \begin{abstract}
We introduce Brauer algebras associated to complex reflection groups
of type $G(m,p,n)$, and study their representation theory via Clifford
theory.  In particular, we determine the 
decomposition numbers of these algebras in characteristic zero.
\end{abstract}
 
 \section*{Introduction}
 
The symmetric and general linear groups satisfy a double centraliser
property over tensor space.  This relationship is known as Schur--Weyl
duality and allows one to pass information between the representation
theories of these algebras. The Brauer algebra was defined to play the
role of the symmetric group algebra in a Schur--Weyl duality with the
orthogonal (or symplectic) group.

The original definition of the Brauer algebra has been generalised in
many directions (see for example \cite{birwen, hoBMW, cfw, cly,
  turwall, koikewall}).  In
this paper we regard the classical Brauer algebra, $B_n(\delta)$ as an
enlargement of the symmetric group algebra; in other words it
corresponds to an enlargement of a complex reflection group of type
$G(1,1,n)$. By considering analogous enlargements of other complex
reflection groups, we arrive at the Brauer algebras of type
$G(m,p,n)$.  

The type $G(m,1,n)$ case was studied in \cite{bcd}, where the
decomposition numbers for these algebras are calculated by a reduction
to the type $G(1,1,n)$ case.  In this paper we study the Brauer
algebras of type $G(m,p,n)$.  Using a combination of diagram algebra
techniques, Clifford theory, and Brauer--Humphreys reciprocity, we
calculate the decomposition numbers of these algebras.
 
 We begin in Section 1 by defining the Brauer algebras, $B_{m,p,n}$, of type $G(m,p,n)$ and realise the algebra of type $G(m,1,n)$ as a skew group algebra.  This will allow us to apply the methods of Clifford theory.
We then review the basic representation theory of complex reflection groups which is both required for and motivates the results that follow.  
  
 In Section 3 we begin to study the representation theory of the Brauer algebras of type $G(m,p,n)$.  We deduce when the algebra is quasi-hereditary and give explicit constructions of the standard modules.  We then apply Clifford theory to deduce restriction rules for standard, simple, and projective modules.  We briefly consider restriction to the underlying group algebra using Littlewood--Richardson theory.  
 
Using Clifford theory and the fact that {\rm Hom}-spaces for $B_{m,1,n}$ have nice rotational symmetries, we are able to decompose Hom-spaces for $B_{m,p,n}$.  Combining these results and Brauer--Humphreys' reciprocity, we conclude by determining the decomposition numbers of $B_{m,p,n}$ in terms of those for the classical Brauer algebra (which have been given in terms of Kazhdan--Lusztig polynomials by \cite{marbrauer}).
 
% 
% \begin{Acknowledgements*}
% We would like to thank Maud De Visscher for %valuable input at the beginning of this project.
% help identifying the family of Brauer algebras studied in this paper.
% We would also like to thank Jean Michel and Shona Yu for helpful discussions
% concerning complex reflection groups   %and the organisers of
% and the organisers of   `Lie theory and quantum analogues' and the CIRM institute for   providing support and a stimulating environment during the  conference.  The first author w is grateful for the financial support received from the ANR grant ANR-10-BLAN-0110.
% \end{Acknowledgements*}

 \section{ Brauer algebras of type $G(m,p,n)$}

We fix $k$, an algebraically closed field. %and fix $\xi \in k^{\times}$
%to be a primitive $m$th root of unity.  
Let $m,p,n \in \mathbb{N}$ be
such that $pd=m$ for some $d\in\NN$.  In this section we will define
the Brauer algebras, $\B$, of type $G(m,p,n)$.  We shall show that the
Brauer algebra of type $G(m,p,n)$ is a subalgebra of that of type
$G(m,1,n)$ introduced in \cite[Appendix]{bcd} (where it was called the
unoriented cyclotomic Brauer algebra).
 
\subsection{Definitions} 
Given $n \in \mathbb{N}$ and $\delta=(\delta_0,\delta_p,
\delta_{2p},\ldots,\delta_{(d-1)p}) \in k^{d }$, the {\it Brauer algebra
  of type $G(m,p,n)$}, denoted by $\B$, is a finite dimensional
associative $k$-algebra generated by certain Brauer diagrams.  A {\it
  diagram} consists of a {\it frame} with $n$ distinguished points on
the northern and southern boundaries, which we call {\it nodes}.  Each
node is joined to precisely one other by a strand; strands connecting
the northern and southern edge will be called \emph{through-strands}
and the remainder (\emph{northern or southern}) \emph{arcs}.  There may also be closed loops inside
the frame, those diagrams without closed loops are called
\emph{reduced} diagrams.

Each strand is labelled by an element of the cyclic group $\ZZ/m\ZZ$;
we require the additional restriction that the total sum over the
labels is a multiple of $p$.  When drawing diagrams we will adopt the
convention that unlabelled arcs have label $0$. Two diagrams are
equivalent if the 
strands connect the same pairs of nodes and have the same labels.  As
a vector space, $\B$ is the $k$-span of the reduced diagrams.  Figure
\ref{xyelts} gives an example of two such elements in $B(6,3,6)$.

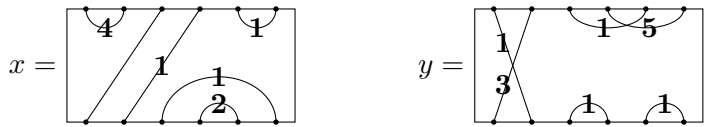
\begin{figure}[ht]
\centerline{
$x=$
\begin{minipage}{45mm}\begin{tikzpicture}[scale=0.5]
  \draw (0,0) rectangle (6,3);
  \foreach \x in {0.5,1.5,...,5.5}
    {\fill (\x,3) circle (2pt);
     \fill (\x,0) circle (2pt);}
  \begin{scope} 
     \draw (2.5,3) -- (0.5,0);     \draw (3.5,3) -- (1.5,0);
%%%%arcs
    \draw (0.5,3) arc (180:360:0.5 and 0.5);
    \draw (2.5,0) arc (180:360:1.5 and -1.2);    \draw (3.5,0) arc (180:360:0.5 and -0.5);    \draw (4.5,3) arc (180:360:0.5 and 0.5);
                                      \draw (1,2.5) node {$ \textbf{4}$};\draw (2.5,1.5) node {$ \textbf{1}$};\draw (4,0.5) node {$ \textbf{2}$};\draw (4,1.2) node {$ \textbf{1}$};\draw (5,2.5) node {$ \textbf{1}$};
  \end{scope}
\end{tikzpicture}\end{minipage}
$y=$
\begin{minipage}{45mm}
\begin{tikzpicture}[scale=0.5]
  \draw (0,0) rectangle (6,3);
  \foreach \x in {0.5,1.5,...,5.5}
    {\fill (\x,3) circle (2pt);
     \fill (\x,0) circle (2pt);}
  \begin{scope} 
     \draw (0.5,3) -- (1.5,0);     \draw (1.5,3) -- (0.5,0);
     \draw (2.5,0) arc (180:360:0.5 and -0.5);  \draw (3.5,3) arc (180:360:1 and 0.5);\draw (2.5,3) arc (180:360:1 and 0.5); \draw (4.5,0) arc (180:360:0.5 and -0.5);
        \draw (3,0.5) node {$ \textbf{1}$};        \draw (5,0.5) node {$ \textbf{1}$};          \draw (4.6,2.5) node {$ \textbf{5}$};    \draw (3.4,2.5) node {$ \textbf{1}$};          \draw (0.75,2.1) node {$ \textbf{1}$};   \draw (0.75,1) node {$ \textbf{3}$};    
   \end{scope}
\end{tikzpicture}\end{minipage}}
 \caption{Two elements in $B_{6,3,6}(\delta)$}
\label{xyelts}
\end{figure}

Given $x,y \in B_{m,p,n}$, we define the product $x \cdot y$ to be the
diagram obtained by concatenation of $x$ above $y$, where we identify
the southern nodes of $x$ with the northern nodes of $y$ and then
ignore the section of the frame common to both diagrams.  

The label of each strand, $s$, in the concatenated diagram, is then
the sum of the labels of the strands it is composed from.  The product of two
diagrams may contain a closed loop: if this loop is labelled by
$ip\in\ZZ/m\ZZ$ then the diagram is set equal to $\delta_{ip}$ times the
same diagram with the loop removed; if the label is not divisible by
$p$, we set the product to be zero.

\begin{eg}
The product $x \cdot y$  of the elements in Figure \ref{xyelts} is
given in Figure \ref{2}.  The product 
$y^2=0$ as it results in the removal of a closed loop labelled by $2$
(when reduced mod $6$), which is not divisible by 3.

\begin{figure}[ht]
\centerline{
\begin{minipage}{37mm}\begin{tikzpicture}[scale=0.5]
  \draw (0,0) rectangle (6,3);
  \foreach \x in {0.5,1.5,...,5.5}
    {\fill (\x,3) circle (2pt);
     \fill (\x,0) circle (2pt);}
  \begin{scope} 
     \draw (2.5,3) -- (0.5,0);     \draw (3.5,3) -- (1.5,0);
     \draw (0.5,3) arc (180:360:0.5 and 0.5);
    \draw (2.5,0) arc (180:360:1.5 and -1.2);    \draw (3.5,0) arc (180:360:0.5 and -0.5);    \draw (4.5,3) arc (180:360:0.5 and 0.5);
                                      \draw (1,2.5) node {$ \textbf{4}$};\draw (2.5,1.5) node {$ \textbf{1}$};\draw (4,0.5) node {$ \textbf{2}$};\draw (4,1.2) node {$ \textbf{1}$};\draw (5,2.5) node {$ \textbf{1}$};
 \end{scope}
   \draw (0,0) rectangle (6,-3);
  \foreach \x in {0.5,1.5,...,5.5}
    {\fill (\x,-3) circle (2pt);
     \fill (\x,0) circle (2pt);}
  \begin{scope} 
     \draw (0.5,0) -- (1.5,-3);     \draw (1.5,0) -- (0.5,-3);
     \draw (2.5,-3) arc (180:360:0.5 and -0.5);  \draw (3.5,0) arc (180:360:1 and 0.5);\draw (2.5,0) arc (180:360:1 and 0.5); \draw (4.5,-3) arc (180:360:0.5 and -0.5);
        \draw (3,-2.5) node {$ \textbf{1}$};        \draw (5,-2.5) node {$ \textbf{1}$};          \draw (4.6,-.5) node {$ \textbf{5}$};    \draw (3.4,-.5) node {$ \textbf{1}$};           \draw (0.75,-0.9) node {$ \textbf{1}$};   \draw (0.75,-2) node {$ \textbf{3}$};    
  \end{scope}
\end{tikzpicture}\end{minipage}
= \  $\delta_3$
\begin{minipage}{33mm}\begin{tikzpicture}[scale=0.5]
   \draw (0,0) rectangle (6,3);
  \foreach \x in {0.5,1.5,...,5.5}
    {\fill (\x,3) circle (2pt);
     \fill (\x,0) circle (2pt);}
  \begin{scope} 
     \draw (2.5,3) -- (1.5,0);     \draw (3.5,3) -- (0.5,0);
     \draw (0.5,3) arc (180:360:0.5 and 0.5);
    \draw (4.5,0) arc (180:360:0.5 and -0.5);    \draw (3.5,0) arc (180:360:-0.5 and -0.5);    \draw (4.5,3) arc (180:360:0.5 and 0.5);
   \draw (2.2,2.3) node {$ \textbf{1}$};                                      \draw (1,2.5) node {$ \textbf{4}$};\draw (1.2,0.8) node {$ \textbf{4}$};\draw (3,0.5) node {$ \textbf{1}$};\draw (5,0.5) node {$ \textbf{1}$};\draw (5,2.5) node {$ \textbf{1}$};
 \end{scope}
 \end{tikzpicture}\end{minipage}
 } \caption{The product  $x \cdot y$  }
\label{2}
\end{figure}
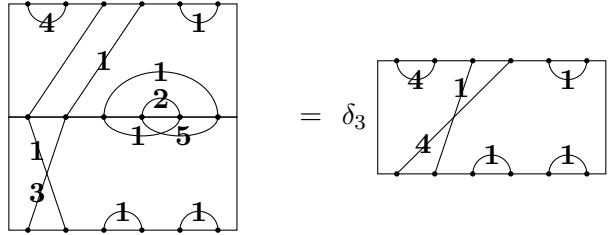

\end{eg}

We will need to speak of certain elements of the algebra with great
frequency. The elements  $s_{i,j}$, $t_i^k$,  $s_{i,j}^\ast$, and $e_{i,j}$ (for $i,j\leq
n$) are indicated in Figure \ref{tare}, where the nodes are numbered
in increasing order from left to right by $1$ up to $n$ on the
northern edge, and $\bar{1}$ up to $\bar{n}$ on the southern edge.
 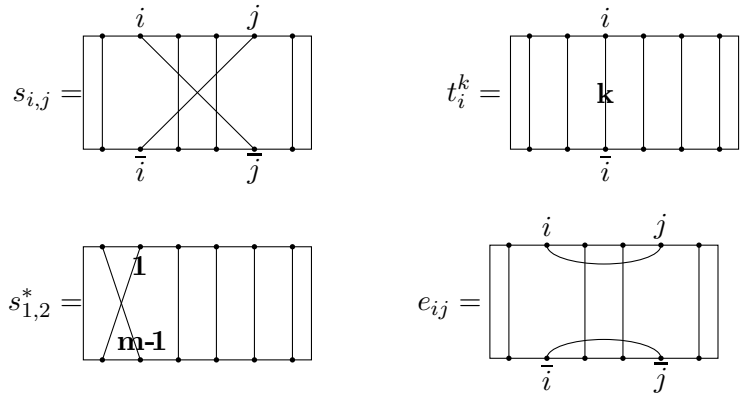
\begin{figure}[ht]
 \begin{align*}
s_{i,j}=&\begin{minipage}{44mm}\begin{tikzpicture}[scale=0.5]
  \draw (0,0) rectangle (6,3);
  \foreach \x in {0.5,1.5,...,5.5}
    {\fill (\x,3) circle (2pt);
     \fill (\x,0) circle (2pt);}
    \draw (4.5,3.5) node {$j$};
    \draw (4.5,-0.5) node {$\overline{j}$};
    \draw (1.5,3.5) node {$i$};
    \draw (1.5,-0.5) node {$\overline{i}$};
  \begin{scope}%[thick]
    \draw (0.5,3) -- (0.5,0);
    \draw (5.5,3) -- (5.5,0);
        \draw (1.5,3) -- (4.5,0);
            \draw (4.5,3) -- (1.5,0);
    \draw (2.5,3) -- (2.5,0);
    \draw (3.5,3) -- (3.5,0);
   \end{scope}
\end{tikzpicture}\end{minipage} \quad
%%%
t_{i}^k=\begin{minipage}{34mm}\begin{tikzpicture}[scale=0.5]
  \draw (0,0) rectangle (6,3);
  \foreach \x in {0.5,1.5,...,5.5}
    {\fill (\x,3) circle (2pt);
     \fill (\x,0) circle (2pt);}
       \draw (2.5,3.5) node {$i$};
              \draw (2.5,-.5) node {$\overline{i}$};
                    \draw (2.5,1.5) node {${\textbf{k}}$};
  \begin{scope} 
    \draw (0.5,3) -- (0.5,0);
    \draw (5.5,3) -- (5.5,0);
        \draw (1.5,3) -- (1.5,0);
            \draw (4.5,3) -- (4.5,0);
             \draw (2.5,3) -- (2.5,0);
    \draw (3.5,3) -- (3.5,0);
   \end{scope}
\end{tikzpicture}\end{minipage}
\\ 
%%%
s^\ast_{1,2}=&\begin{minipage}{44mm}\begin{tikzpicture}[scale=0.5]
  \draw (0,0) rectangle (6,3);
  \foreach \x in {0.5,1.5,...,5.5}
    {\fill (\x,3) circle (2pt);
     \fill (\x,0) circle (2pt);}
      \draw (1.33,2.5) node {${\textbf{  1}}$};
  \draw (1.4,0.5) node {${\textbf{ m-\!1}}$};
  \begin{scope}%[thick]
    \draw (0.5,3) -- (1.5,0);
    \draw (1.5,3) -- (0.5,0);
        \draw (5.5,3) -- (5.5,0);
            \draw (4.5,3) -- (4.5,0);
    \draw (2.5,3) -- (2.5,0);
    \draw (3.5,3) -- (3.5,0);
   \end{scope}
\end{tikzpicture}\end{minipage} 
 e_{ij}= \begin{minipage}{34mm}\begin{tikzpicture}[scale=0.5]
  \draw (0,0) rectangle (6,3);
  \foreach \x in {0.5,1.5,...,5.5}
    {\fill (\x,3) circle (2pt);
     \fill (\x,0) circle (2pt);}
    \draw (4.5,3.5) node {$j$};
    \draw (4.5,-0.5) node {$\overline{j}$};
    \draw (1.5,3.5) node {$i$};
    \draw (1.5,-0.5) node {$\overline{i}$};
  \begin{scope} 
    \draw (0.5,3) -- (0.5,0);
    \draw (5.5,3) -- (5.5,0);
           \draw (1.5,3) arc (180:360:1.5 and 0.5);
    \draw (4.5,0) arc (0:180:1.5 and .5);
    \draw (2.5,3) -- (2.5,0);
    \draw (3.5,3) -- (3.5,0);
   \end{scope}
\end{tikzpicture} \end{minipage}\end{align*}
 \caption{The elements $s_{i,j}$, $t_i^k$, and $s_{i,j}^\ast$ and $e_{i,j}$  }
\label{tare} 
\end{figure}

\begin{rem}
The $p=1$ case was first  studied in the Appendix to
\cite{bcd}.  There it is christened the \emph{un-oriented cyclotomic
  Brauer algebra}; this algebra is not the (oriented) cyclotomic
Brauer algebra studied in \cite{amr} and elsewhere.  Both the oriented
and un-oriented cyclotomic Brauer algebras are specialisations of the
BMW algebra.  However, it is only the un-oriented algebra which has a
family of subalgebras which can be studied by analogy with the complex
reflection groups of type $G(m,p,n)$.
\end{rem}
 
\subsection{Clifford theory I}\label{deltapa}
Consider the algebra $\Bone$, as defined above.  Let $p|m$ and
specialise the parameter $\delta\in k^m$ so that $\delta_i$ is
zero for any index $i$ that is not congruent to zero modulo $p$,
i.e. take $$\delta=(\delta_0, \ldots, 0, \delta_p, 0,\ldots,
0,\delta_{2p},0, \ldots, 0, \delta_{p(d-1)},0, \ldots, 0)\in k^m.$$

Take the subspace of $B_{m,1,n}$ (with parameter as above) spanned by
all diagrams whose labels sum to a multiple of $p$.  Multiplication is
inherited from that in $B_{m,1,n}$; our choice of parameter ensures
that any closed loops removed are labelled by a multiple of $p$
(otherwise the product is zero) and therefore the diagram obtained by
their removal still lies in the same subspace.  Therefore this
subspace is in fact a subalgebra, and is clearly
isomorphic to $B_{m,p,n}$.  Throughout this paper, we shall only consider $B_{m,1,n}$ for the
  parameter as above.

  Let $\ZZ/p\ZZ$ act via the   $k$-algebra automorphism  of  $B_{m,1,n}$ given by conjugation by $t_1^d$. 
This maps $B_{m,p,n}$  onto $B_{m,p,n}$.  We have the following theorem.

\begin{thm}
The algebra $B_{m,1,n}$ (with parameter $\delta$ as above)  is the skew group algebra 
$$B_{m,1,n} =   B_{m,p,n} \rtimes   \ZZ/p\ZZ  = \left\{	\sum_{z \in  \ZZ/p\ZZ} d_z z	: d\in B_{m,p,n}		 \right\}$$
with linear multiplication given by the concatenation action:
$ z d = (zdz^{-1})	z	.		$

\end{thm}
\begin{proof}
This is similar to the the group algebra case.  The natural diagram basis of $B_{m,1,n}$ can be partitioned into $p$ distinct sets, $\mathcal{B}_0$, $\mathcal{B}_1, \ldots, \mathcal{B}_{p-1}$, (of equal cardinality)  each consisting of the diagrams whose sum over the labels is congruent to $0,1, \ldots p-1$ modulo $p$ respectively.  The algebra $B_{m,p,n}$ has basis given by $\mathcal{B}_0$, as seen above. 

Left and right multiplication by $t_1$ both define bijections from $\mathcal{B}_i$ to  $\mathcal{B}_{i+1}$.  Using this to rewrite diagrams   as $t_1^i d'$ or as $ d''t_1^j$ for some $d',d''\in \mathcal{B}_0$, one can check that the multiplication on $  B_{m,p,n} \rtimes   \ZZ/p\ZZ$ is equivalent to the multiplication on $B_{m,1,n}$ (with $\delta$ as above).
 \end{proof} 
 
 This result means that we will later be able to apply methods from Clifford theory (see \cite{ramram}).
 
%   We may
%  therefore use the short exact sequence
%$$0 \to B_{m,p,n} \to B_{m,1,n} \to \ZZ/p\ZZ \to 0$$
%to study $B_{m,p,n}$ via Clifford theory.  This works entirely analogously to the group algebra case and the technicalities are discussed in \cite[Appendix]{ramram}.

%
%  It is easy to see that we have a short
%exact sequence of algebras:
%$$ 0 \to B_{m,p,n} \to B_{m,1,n} \to k(\ZZ/p\ZZ) \to 0, $$ which
%allows us to study the algebras $B_{m,p,n}$ via Clifford theory.
%Throughout this paper, we shall only consider $B_{m,1,n}$ for the
%parameter as above.

\subsection{Generators for subalgebras} Just as for the Brauer
algebra, it follows from the 
definitions that $\Bone$ is generated by $s_{i,i+1}$, $t_1$ and
$e_{1,2}$.  The group algebra of type $G({m,p,n})$  can be identified with the
subalgebra of $\Bone$ generated by $s_{1,2}^\ast, s_{i,i+1}$ and $t_1^p$; for
$n>2$ the subalgebra $\B$ of $\Bone$ is generated by 
$s_{1,2}^\ast, s_{i,i+1}$, $t_1^p$ and $e_{1,2}$ for $1\leq i \leq n-1$.
 
\subsection{Cyclotomic parameters} We have defined the algebra
$B_{m,p,n}$ in terms of $\delta =(\delta_0, \delta_p, \ldots,
\delta_{(d-1)p}) \in k^{d}$.  It is shown in \cite{bcd} that the following
cyclotomic functions of these parameters govern the representation
theory of the algebra $B_{m,1,n}$ under the assumption that $m$ is invertible in $k$.
 \begin{defn}
For each $0\leq r \leq d-1$ we define the $r$th   cyclotomic
parameter to be
\begin{align*}
\overline{\delta}_r = 
\frac{1}{m} \sum_{i=0}^{d-1} \xi^{ipr} \delta_{ip}.
\end{align*} where $\xi \in k^{\times}$
 is  a primitive $m$th root of unity
\end{defn}

%% \subsection{Automorphisms}
%%Let $\tau$ be the $k$-algebra automorphism on $B_{m,p,n}$ which is defined as conjugation by $t_1$; on the generators $\tau(s_{1,2})=s_{1,2}^\ast$, $\tau(t_{1})=t_{1}$ and $\tau(s_{i,i+1})=s_{i,i+1}$ for all $2\leq i\leq n$.  
%
%Let $\sigma$ denote the $k$-algebra automorphism on $B_{m,1,n}$ which is defined on generators by $\sigma(t_1)=\xi^\ast t_1^\ast$ and $\sigma(s_{i,i+1})=s_{i,i+1}$ for all $1\leq i\leq n$.  
%

 \section{Reflection groups of type $G(m,p,n)$}

We have already assumed that $k$ is an algebraically closed
field. Henceforth we shall also assume that   $k$ is of characteristic zero and we shall fix  $\xi \in k^{\times}$,
 a primitive $m$th root of unity.  
% Using the short exact sequence $$0 \to G({m,p,n})
%\to G({m,1,n}) \to k(\ZZ/p\ZZ) \to 0$$
The group algebra of the complex reflection group, $G{(m,1,n)}$,  is the \emph{skew group algebra}
$$ G({m,1,n}) =    G({m,p,n}) \rtimes   \ZZ/p\ZZ,  $$
this comes from taking the semidirect product of the two groups.  We shall study $G({m,p,n})$ via
Clifford theory. The results in this section can be found in
 \cite[Section 2.3]{MM}.

 \subsection{Type $G(m,1,n)$ combinatorics}   A partition is a finite
weakly-decreasing sequence of non-negative integers. An $m$-{\it
  partition} of $n$ is an $m$-tuple of partitions $\lambda
=(\lambda^0, \ldots , \lambda^{m-1})$ such that $\sum_{i=0}^{m-1}
|\lambda^i|=n$ (where $|\lambda^i|$ denotes the sum of the parts of
the partition $\lambda^i$).  %There is a natural ordering on this set,
%called \emph{dominance}, defined in \cite{bcd}.
We let $\Lambda(m,1,n)$ denote the set of all $m$-partitions of $n-2l$
for $l\leq n/2$; we let $\Lambda_0(m,1,n)$ denote the subset where
$l=0$.  %Given an $m$-partition $\lambda$ we associate the
%$m$-composition
%$$|\lambda|=(|\lambda^0|,|\lambda^1|, \ldots , |\lambda^{m-1}|).$$ For
%%an $m$-composition of $n$, $\omega$, we define another $m$-composition
%$[\omega]$ by
%$$[\omega]=([\omega_0], [\omega_1], \ldots, [\omega_{m-1}]=n)$$
%where $[\omega_r]=\sum_{i=0}^r \omega_i$ for $0\leq r < m$.
%For an $m$-partition $\lambda$ we define $[\lambda]=[|\lambda|]$.

%%%%%%%%%%%%%%%%%%%%%%%%%% needed

Let $\lambda$ be an $m$-partition of $n$.  A $\lambda$-tableau % (or $m$-tableau)
 is a
bijection $\mathfrak{t}:\lambda \to \{1,2...,n\}$, which we consider
as an $m$-tuple $\mathfrak{t} = (\mathfrak{t}^0, \ldots,
\mathfrak{t}^{m-1})$ of labelled tableaux where $\mathfrak{t}^s$ is a
$\lambda^s$-tableau for each $s$; the tableaux $\mathfrak{t}^s$ are
the components of $\mathfrak{t}$.  We say a tableau, $\ft$, is \emph{standard} if the entries in the component tableaux are increasing along the rows and columns.  
We let $\mathcal{T}_{\lambda}$ denote the set of standard $\lambda$-tableaux.
%Let $\mathfrak{t}^\lambda$ be the
%$\lambda$-tableau obtained by putting the integers $1,\ldots, n$ in
%order into the boxes in the Young diagram of $\lambda$ from left to
%right down successive rows.

For $\mathfrak{t}$ a tableau, we set $\mathfrak{t}(i)=s$ if the
integer $i$ appears in $\mathfrak{t}^s$. Let $1\leq i < j\leq n$, we
define the axial distance, $a(i,j)$, as follows: if
$\mathfrak{t}(i)\neq \mathfrak{t}(j)$ then $a(i,j)=\infty$ (so that
$1/a(i,j)=0$); if $\mathfrak{t}(i) = \mathfrak{t}(j)$ and $i$ occurs
in row $i_0$ and column $i_1$ and $j$ occurs in row $j_0$ and column
$j_1$, then $a(i,j)=(i_0-i_1) -(j_0-j_1)$.

If $\mathfrak{t}$ is a $\lambda$-tableau and $w \in \Sigma_n$ let
$w\mathfrak{t}$ be the tableau obtained from $\ft$ by replacing each entry
in $\mathfrak{t}$ by its image under $w$.
Let $\ft \in \mathcal{T}_\lambda$,
we set 
$\ft_{i\leftrightarrow i+1}$ equal to $s_{i, i+1}\mathfrak{t}$ if this is still a standard $\lambda$-tableau, and 0 otherwise. 
 
\label{spechth} 
 
 \begin{prop}%[Section 5 of \cite{gl}]
The algebra $kG({m,1,n})$ has simple modules indexed by the poset
$\Lambda_0(m,1,n)$.  For a given $m$-partition $\lambda$ of $n$, the
simple   module $\overline{\mathbf{S}(\lambda)}$
has a basis given by the set of standard $\lambda$-tableaux.
 With respect to this
basis the generators act as follows
\begin{align*} \overline{\rho_{\lambda}}(t_1)\ft = \xi^{\mathfrak{t}(1)}\ft,	
\quad\quad\quad
 \overline{\rho_{\lambda}}(s_{i,i+1})\ft = \frac{1}{a(i,i+1)}  \ft  + \left(1 + \frac{1}{a(i,i+1)}\right)  \ft_{i \leftrightarrow i+1}  	
\end{align*}
%with the obvious extension of  notation from $\mathcal{T}_\lambda$.
  \end{prop}

 \subsection{Type $G(m,p,n)$ combinatorics}\label{G2n}\label{typeDspecht}
Let $pd=m$ and let $\sigma$ be a distinguished generator of
$\ZZ/p\ZZ$.  There is a natural action of the cyclic group $\ZZ/p\ZZ$
on the poset $\Lambda(m,1,n)$ given by permutation of the indices.
This extends to an action on tableaux by setting
  $$\sigma :(\mathfrak{t}^0,\mathfrak{t}^1,\ldots,\mathfrak{t}^{m-1} )
\mapsto (\mathfrak{t}^{m-d},\mathfrak{t}^{1-d}, \ldots
,\mathfrak{t}^{m-1-d}),$$ we denote
$\sigma(\mathfrak{t})=\mathfrak{t}^\sigma$. % (and extend this notation
%to standard tableaux and hence to $\mathcal{D}_\lambda$).

For $\lambda \in \Lambda(m,1,n)$ let $\Stab(\lambda)$ denote the
stabiliser of $\lambda$ under the permutation action.  We have
$\Stab_{\ZZ/p\ZZ}(\lambda)=\langle \sigma^t\rangle$ and $0\leq r <
p/t$.  We let $\Lambda(m,p,n)$ denote the set of pairs consisting of a representative of a
$\ZZ/p\ZZ$-orbit on $\Lambda(m,1,n)$ and an integer $0\leq r <
p/t$.  We let $\Lambda_0(m,p,n)$ denote the subset where $l=0$.

\begin{eg}
We have that $\Lambda(2,1,2)=\{ (\o, \o), (\o,2), (\o,1^2), (2,\o),
(1^2,\o)\}$.  There is a unique element, $(\o, \o)$, with non-trivial
stabiliser $\langle \sigma^1\rangle = \ZZ/2\ZZ$.  Therefore
$\Lambda(2,2,n)$ has four elements and (picking a set of orbit
representatives) is equal to the set $\{ (\o, \o)^0, (\o, \o)^1,
(2,\o), (1^2,\o)\}$.
\end{eg}

\subsection{Simple modules   for $G({m,p,n})$}\label{spechth2} We now
give the construction, via Clifford theory, of the simple modules for
$G({m,p,n})$.  We do not go into much detail here, and instead refer
to \cite{MM}.
%  We have a short exact sequence of algebras
%$$0 \to G({m,p,n})\to G({m,1,n})\to k(\ZZ/p\ZZ)\to 0.$$

Simple modules for $G({m,p,n})$ are labelled by a representation of
$\Stab(\lambda)\leq \ZZ/p\ZZ$ (given by an integer $0\leq r < p/t$)
and a representation of $G({m,1,n})$ (given by an $m$-partition).  
%We have a chain of subgroups $H_{m,1,|\lambda|} \leq \langle \sigma^t
%\rangle \wr H_{m,1,|\lambda|} \leq H_{m,1,n}$. One first wishes to
%induce to $H_{m,1,|\lambda|} \leq \langle \sigma^t \rangle \wr
%H_{m,1,|\lambda|}$, and then decompose this module into a direct sum
%of indecomposable modules.   
Recall that $\Stab(\lambda) = \langle \sigma^t \rangle \leq \ZZ/p\ZZ$.
 By Clifford theory, we have that  $\overline{\textbf{S}(\lambda)} \!\!\downarrow =
\oplus_{0\leq r<p/t} \mathbf{S}(\lambda^r)$, where
$$\mathbf{S}(\lambda^r)=\ker(\sigma^t -
\xi^{dtr})\overline{\mathbf{S}(\lambda)}$$ for $0\leq r < p/t$.   We  
let $$p_r= \frac{t}{p} \sum_{0\leq i <p/t}\xi^{-idtr}\sigma^{it}$$
denote the projection onto this subspace.

Take as representatives of the $ \langle \sigma^t \rangle$-orbits the
%$x$ such that 
$\mathfrak{t}\in\mathcal{T}_{\lambda}^0$
 where $\mathcal{T}_{\lambda}^0$
is the set of standard $\lambda$-tableaux with $\mathfrak{t}(1)< td$.  Take
the subspace spanned by tableaux in $\mathcal{T}_{\lambda}^0$
and apply the projection $p_r$, this provides a basis of
$\mathbf{S}(\lambda^r)$ (in the case that $r=0$ this is the average of the
$ \langle \sigma^t \rangle$-orbit).  Setting $\ft^r = p_r(\ft)$, we then
get formulae for the action of the generators of $G({m,p,n})$ on
$\mathbf{S}(\lambda^r)$ as follows: %.  We have that
 $\overline{\rho_\lambda}(t)p_r = p_{r+1}\overline{\rho_\lambda}(t)$, and so
 %$\overline{\rho_\lambda}(t)p_r = p_{dr}\overline{\rho_\lambda}(t)$, and so
$$\rho_{\lambda,r}(t^p)\ft^r = \xi^{pr \mathfrak{t}(1)}\ft^r, \quad
\rho_{\lambda,r}(s_1^\ast)\ft^r = 	\xi^{
  r(\ft(1)-\ft(2))}\rho_{\lambda,r}(s_1)\ft^r,								$$  
 $$\rho_{\lambda,r}(s_i)\ft^r = \frac{1}{a(i,i+1)}  \ft^r  + \left(1 +
\frac{1}{a(i,i+1)}\right)  \ft^r_{i \leftrightarrow i+1} . 	$$

 \section{Brauer algebras of type $G(m,p,n)$} 
 Unless otherwise stated, let $k$ denote an
 algebraically closed field of characteristic zero.  Fix $\xi$ to be a primitive $m$th root of unity.
 In this section we study $B_{m,p,n}$ via Clifford theory.  By
 verifying the first two conditions of a tower of recollement, we
 deduce conditions under which the algebra is quasi-hereditary; we
 leave it as an exercise for the reader to check that the algebra is a
 tower of recollement (in the sense of \cite{cmpx}) by checking that it obeys the remaining conditions.

\subsection{Highest weight theory }

Let $n\geq 2$. Suppose first that $\delta\neq \bf{0}\in k^{d}$ and fix a
$\delta_{ip} \neq  0$ for some $0\leq i<d$. We then define the
idempotent $e_{n-2}=\frac{1}{\delta_{ip}}t_{n-1}^{ip}e_{n-1,n}$ as
illustrated in Figure \ref{idem}. Note that it is a scalar multiple of
a diagram with $n-2$ through-strands.  If $\delta =0$ and $n\geq 3$
then we define $e_{n-2}$ to be the idempotent $e_{n-1,n}e_{n-2,n-1}$,
as illustrated in Figure \ref{idem}.
\begin{figure}[ht]
\centerline{
$   \dfrac{1}{\delta_{ip}}$ \ \begin{minipage}{54mm}\begin{tikzpicture}[scale=0.5]
  \draw (0,0) rectangle (6,3);
  \foreach \x in {0.5,1.5,...,5.5}
    {\fill (\x,3) circle (2pt);
     \fill (\x,0) circle (2pt);}
    \draw (5,2.4) node {$\textbf{ip}$};
%    \draw (2.5,-0.5) node {$r-2l$};
  \begin{scope}%[thick]
    \draw (3.5,3) -- (3.5,0);    \draw (0.5,3) -- (0.5,0);    \draw (1.5,3) -- (1.5,0);    \draw (2.5,3) -- (2.5,0);
           \draw (4.5,3) arc (180:360:.5 and 0.5);
   \draw (4.5,0) arc (180:360:.5 and -0.5);
%           \draw (5.5,3) arc (180:360:.5 and 0.5);
%   \draw (5.5,0) arc (180:360:.5 and -0.5);
   \end{scope}
\end{tikzpicture} \end{minipage}\quad
 \begin{minipage}{34mm}\begin{tikzpicture}[scale=0.5]
  \draw (0,0) rectangle (6,3);
  \foreach \x in {0.5,1.5,...,5.5}
    {\fill (\x,3) circle (2pt);
     \fill (\x,0) circle (2pt);}
%    \draw (2.5,-0.5) node {$r-2l$};
  \begin{scope}%[thick]
    \draw (3.5,3)  to [out=-90,in=90]  (5.5,0);    \draw (0.5,3) -- (0.5,0);    \draw (1.5,3) -- (1.5,0);    \draw (2.5,3) -- (2.5,0);
           \draw (4.5,3) arc (180:360:.5 and 0.5);
   \draw (4.5,0) arc (180:360:-.5 and -0.5);
%           \draw (5.5,3) arc (180:360:.5 and 0.5);
%   \draw (5.5,0) arc (180:360:.5 and -0.5);
   \end{scope}
\end{tikzpicture} \end{minipage}  }
 \caption{The idempotent $e_{n-2}$  (for $n=6$) in the cases that $\delta\neq0$, $\delta=0$ respectively.}
\label{idem}
\end{figure}
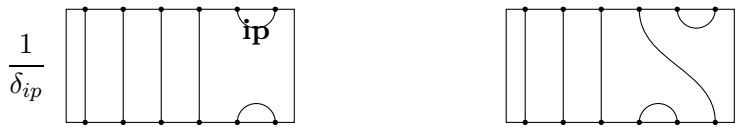

  A tower of recollement was defined in \cite{cmpx} to be a family of
  algebras (with idempotents) satisfying six conditions (A1--6). It is
  easy to see that
\begin{equation}\label{eAe}
e_{n-2}B_{m,p,n}e_{n-2}\cong B_{m,p,n-2}
\end{equation}
and that 
\begin{equation}\label{eAe2}
B_{m,p,n}/B_{m,p,n}e_{n-2} B_{m,p,n}\cong kG({m,p,n}).\end{equation} 
For the latter isomorphism, note that the lefthand-side has a basis consisting of the diagrams with no arcs.
Therefore we have the following
\begin{thm}\label{A2'}
Let $k$ be a field of characteristic $\cha(k)\geq 0$.  Let $m,n\in\NN$, and
$\delta \in k^m$.  If $n$ is even suppose $\delta \neq 0 \in k^m$.
The algebra $B_{m,p,n}(\delta)$ is quasi-hereditary if and only if
$\cha(k)>n$ and $\cha(k)\!\!\! \not| m$, or $\cha(k)=0$. 
\end{thm}
 
  We leave it to the reader to verify   the remaining tower
  conditions using classical tower arguments (see \cite{cddm},
  \cite{cdm}) and Clifford theory.

\subsection{The standard modules of $B_{m,1,n}$}\label{cliff}
Recall our assumption on the parameter $\delta\in k^m$ from Section
\ref{deltapa}.  By \cite[Theorem 3.1.2]{bcd}, the algebra $B_{m,1,n}$ is an iterated inflation of
the group algebras $G({m,1,{n-2l}})$ along vector spaces $V_l$ spanned
by all possible $(m,n,l)$-tangles.  An $(m,n,l)$-tangle has $l$ arcs
denoted by $(i_p,j_p)$ (for $p=1,\ldots, l$) where $i_p$ (resp. $j_p$)
is the left (resp. right) vertex of the arc, and $n-2l$ free
lines. Each arc has a label given by an element $r\in \ZZ/m\ZZ$.  For
example a $(5,7,2)$-tangle is depicted in Figure \ref{dangle}.

\begin{figure}[ht]
\centerline{
 \begin{minipage}{34mm}\begin{tikzpicture}[scale=0.5]
   \foreach \x in {1.5,...,6.5,7.5}
    {\fill (\x,3) circle (2pt);
     }
%    \draw (2.5,3.5) node {$\tiny{r-2l}$};
%    \draw (2.5,-0.5) node {$r-2l$};
  \begin{scope}%[thick]
    %\draw (7.5,3) -- (7.5,1);  
     \draw (7.5,3) to [out=-90,in=90] (3.5,1);
     \draw (1.5,3) -- (1.5,1);    \draw (2.5,3) -- (2.5,1);
           \draw (3.5,3) arc (180:360:1 and 0.5);
           \draw (4.5,3) arc (180:360:1 and 0.5);
               \draw (4.4,2.5) node {$\textbf{1}$};               \draw (5.6,2.5) node {$\textbf{3}$};
   \end{scope}
\end{tikzpicture} \end{minipage}  }
 \caption{A $(5,7,2)$-dangle}
\label{dangle}
\end{figure}
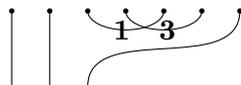

   We therefore have the following theorem:
%\begin{thm}\label{barstand}
% The algebra $B_{m,1,n}$ is cellular, with standard modules indexed by
% $\Lambda({m,1,n})$.  For a given $m$-partition, $ \lambda$, of
% $n-2l$, the standard module $\overline{\Delta(\lambda )}$ has basis:
% \begin{align*}
% \{ v \otimes  x     :  v \in V_l, x \in \overline{\mathbf{S}(\lambda ) }\}
% \end{align*}
%
%\end{thm}
\begin{thm}\label{barstand}
 The algebra $B_{m,1,n}$ has standard
  modules indexed by
 $\Lambda({m,1,n})$.  For a given $m$-partition, $ \lambda$, of
 $n-2l$, we have the standard module  
$$\overline{\Delta(\lambda )}\cong V_l \otimes_k  \overline{\mathbf{S}(\lambda ) }.$$
 \end{thm}

  The action of a diagram $X\in B_{m,1,n}$ on $v\otimes x \in
\overline{  \Delta(\lambda )}$ is given as follows. Apply the diagram $X$ to the
  $(m,n,l)$-tangle $v$. If we obtain more than $l$ arcs, or a closed
  loop labelled by an integer not divisible by $p$, this element is
  sent to zero. Otherwise, we obtain another $(m,n,l)$-tangle $Xv$ and
  a signed permutation $\sigma \in G({m,1, n-2l})$ on the $n-2l$ free
  vertices of $Xv$, we then define $X(v\otimes x)=(Xv)\otimes \sigma
  x$.

% 

%
%\subsection{Clifford theory}
%Recall that $\tau$ is the $k$-algebra automorphism on $B_{m,p,n}$ which is defined as conjugation by $t_1$; on the generators $\tau(s_{1,2})=s_{1,2}^\ast$, $\tau(e_{1,2})=e_{1,2}^\ast$, $\tau(t_{1})=t_{1}$ and $\tau(s_{i,i+1})=s_{i,i+1}$ for all $2\leq i\leq n$.  Recall that $\sigma$ denotes the   $k$-algebra automorphism on $B_{m,p,n}$ which is defined on generators by $\sigma(t_1)=\xi t_1$, $\sigma(e_{1,2})=e_{1,2}$, and $\sigma(s_{i,i+1})=s_{i,i+1}$ for all $1\leq i\leq n$.  
%
%  We have that $\ZZ/p\ZZ$ acts by conjugation on $B_{m,p,n}$.  One can easily check that under this action the \emph{skew group algebra} defined in \cite{ramram} is the algebra $B_{m,1,n}$.  
%We may therefore use the short exact sequence
%$$0 \to B_{m,p,n} \to B_{m,1,n} \to \ZZ/p\ZZ \to 0$$
%to study $B_{m,p,n}$ via Clifford theory.

%
%\subsection{Clifford theory}
% Let $\ZZ/p\ZZ$ act via the   $k$-algebra automorphism, $\tau$, on  $B_{m,p,n}$ given by conjugation by $t_1^d$. %; on the generators $\tau(s_{1,2})=s_{1,2}^\ast$, $\tau(e_{1,2})=e_{1,2}^\ast$, $\tau(t_{1})=t_{1}$ and $\tau(s_{i,i+1})=s_{i,i+1}$ for all $2\leq i\leq n$.  
%This maps $B_{m,p,n}$  onto $B_{m,p,n}$.
%   We may
%  therefore use the short exact sequence
%$$0 \to B_{m,p,n} \to B_{m,1,n} \to \ZZ/p\ZZ \to 0$$
%to study $B_{m,p,n}$ via Clifford theory.  This works entirely analogously to the group algebra case and the technicalities are discussed in \cite[Appendix]{ramram}.
%
%  
%

\subsection{Standard  modules for $B_{m,p,n}$}
%In the previous section we gave a Clifford theoretic construction of
%the indecomposable summands of the restriction of a standard module.
%In this section, we construct the standard modules of $B_{m,p,n}$ via the tower
%formalism and show that these two constructions coincide.
%  
By (\ref{eAe}) and (\ref{eAe2}), we have that the standard modules for $B_{m,p,n}$
are of the form
 $$\Delta_n(\lambda^r)=(B_{m,p,n}/(B_{m,p,n}e_{n-2l-2}B_{m,p,n}))e_{n-2l} \otimes_{B_{m,p,n-2l}} \mathbf{S}(\lambda^r).$$

This module is spanned by the elements $d%e_{n-2l}
 \otimes_{B_{m,p,n-2l}}
\ft^r$ where $\ft^r\in \mathbf{S}(\lambda^r)$ and $d\in B_{m,p,n}$ with precisely $(n-2l)$ through-lines.  By taking elements of
$B_{m,p,n-2l}$ across the tensor product we can just consider diagrams
$d$ with (a) no crossing through-lines (b) only the leftmost
through-line has a non-zero label, (c) this label, $q$, is
strictly less than $p$ (as any diagram $d'\in B_{m,p,n}$ can be written as a product $d'= d\sigma$ for $\sigma \in G(m,p,n)$ and $d$ of the required form).  Of course, these diagrams must still be
elements of $B_{m,p,n}$ and so the northern arcs of the diagram must have labels
totalling $p-q$ modulo $p$.  Figure \ref{modp} contains an example
for type $G(6,3,7)$.

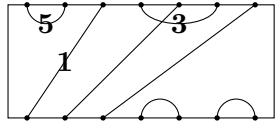
\begin{figure}[ht]
\centerline{
\begin{minipage}{45mm}\begin{tikzpicture}[scale=0.5]
  \draw (0,0) rectangle (7,3);
  \foreach \x in {0.5,1.5,...,6.5}
    {\fill (\x,3) circle (2pt);
     \fill (\x,0) circle (2pt);}
  \begin{scope} 
     \draw (2.5,3) -- (0.5,0);     \draw (4.5,3) -- (1.5,0);    \draw (6.5,3) -- (2.5,0); 
%%%%arcs
    \draw (0.5,3) arc (180:360:0.5 and 0.5);
    \draw (5.5,0) arc (180:360:0.5 and -0.5);   \draw (3.5,0) arc (180:360:0.5 and -0.5);    \draw (3.5,3) arc (180:360:1 and 0.5);
                                      \draw (1,2.5) node {$ \textbf{5}$};\draw (1.5,1.5) node {$ \textbf{1}$} ;\draw (4.5,2.5) node {$ \textbf{3}$};
  \end{scope}
\end{tikzpicture}\end{minipage}}
 \caption{A diagram of type $G(6,3,7)$ satisfying condition (a), (b),
   and (c), above.} 
\label{modp}
\end{figure}

One can then pass the decoration on the left-most strand through the tensor product by noting that $t_1 \ft^r = \ft^{r+1}$ and that $t_1^{p/t} \ft^r = \ft^{r}$, by construction.   Define 
 $V_l(q,p/t)\subset V_l$ to be  the subspace of dangles whose label sum
is congruent to $-q$ modulo $p/t$.

\begin{thm}\label{basisss}
The algebra $B_{m,p,n}$ has standard
modules labelled by $\Lambda(m,p,n)$.  For $\lambda^r \in
\Lambda(m,p,n)$, we have that  
\begin{align*}
 \Delta_n(\lambda^r)\cong\{ v \otimes x : v \in V_l(q,p/t), x \in  %\Delta_{n-2l}
\mathbf{S}(\lambda^{q+r}), 0 \leq q <p/t		\}
\end{align*}
%\begin{align*}
% \Delta_n(\lambda^r)\cong\{ v \otimes x : v \in V_l(q,p/t), x \in  %\Delta_{n-2l}
%\mathbf{S}(\lambda^{q+r}), 0 \leq q <p/t		\}
%\end{align*}
%where $V_l(q,p/t)\subset V_l$ is the space of dangles whose label sum
%is congruent to $ -q$ modulo $p/t$.  
 \end{thm}

\begin{eg}
The modules $\Delta((1,0,1,0)^0)$ and  $\Delta((1,0,1,0)^1)$ for $B_{4,4,4}$ are both 24-dimensional.  %To verify this note that $\overline{\mathbf{S}(1,0,1,0)}$ is 2-dimensional and $V_1$ is $4 \times  {4\choose 2}
%\left(\begin{array}{ccc}0 & 0 & 0 \\0 & 0 & 0 \\0 & 0 & 0\end{array}\right)$ dimensional and therefore $\overline{	\Delta(1,0,1,0)}$ is 48-dimensional.
Let $\ft$ denote the unique element of $\mathcal{T}^0_{(1,0,1,0)}$ and let $v$ be the dangle with a single undecorated arc $(1_p,2_p)$.
Some typical elements of $\Delta((1,0,1,0)^0)$ are 
$$v \otimes \ft^0, \quad t_1^2v \otimes \ft^0,\quad
t_1v \otimes \ft^1, \quad t_1^3v \otimes \ft^1,$$
and 
some typical elements of $\Delta((1,0,1,0)^1)$ are 
$$v \otimes \ft^1, \quad t_1^2v \otimes \ft^1,\quad
t_1v \otimes \ft^0, \quad t_1^3v \otimes \ft^0.$$
In fact, the bases of both modules can be obtained by applying undecorated elements of $G(4,4,2)$ to the elements above (i.e. by permuting the nodes of the dangles).
\end{eg}

%For any $X\in B_{m,p,n}$, any $d$ of the above form, and any element
%$\ft^r\in \mathbf{S}(\lambda^r)$ we define $X(d\otimes x)$ as follows. Place
%the diagram $X$ above 
% $d$.  If we obtain more than $l$ arcs, or a closed loop labelled by
%an integer not divisible by $p$, this element is sent to
%zero. Otherwise, we obtain a diagram
% $Xd$  with labels totalling a multiple of $p$.  If the top
%has label total $q$, then we also  obtain an element
%$\sigma\in G({m,1,n})$ acting on the free $n-2l$ nodes with labels
%totalling  $p-q$. 
%Write $\sigma=t_1^{p-q} \sigma' $ as the product such that
%$\sigma'\in G(m,p,n)$. 
%We define $X(d\otimes \ft^r)=(Xdt_1^{p-q})\otimes \sigma' \ft^r =
%(Xd)\otimes \sigma \ft^{r-q}$. 
%We have the following theorem:

  %
%\begin{eg}
%Consider the $B_{2,1,2}$-module $\Delta_2(\o\o)$.  This module is 2-dimensional with basis $\{T_1^0e_{1,2},T_1^1e_{1,2}\}$.  Taking the projections $p_0$ and $p_1$ we get $B_{2,2,2}$-modules
% $$\Delta_2(\o\o^0)=(1+S){\rm{Span}}_k\{T_1^0e_{1,2},T_1^1e_{1,2}\}$$ and
% $$\Delta_2(\o\o^1)=(1-S){\rm{Span}}_k\{T_1^0e_{1,2},T_1^1e_{1,2}\}$$
%We have that $(1+S)T_1^0e_{1,2}= (1+S)T_1^1e_{1,2}= e_{1,2}$ and $(1-S)T_1^0e_{1,2}= (1-S)T_1^1e_{1,2}= t_1e_{1,2}$.
%\end{eg}

\subsection{Clifford theory II}
We will use Clifford theory techniques to give the decomposition of
the restriction of a standard, simple, or projective module 
from $B_{m,1,n}$ to $B_{m,p,n}$.  

\subsubsection{Standard modules} Let $(i_p,j_p)$ and $(i_q,j_q)$ be two arcs in $v$ with annotations
$l$ and $k$, respectively.  We let $\epsilon_{l,ip}$ denote the
Kronecker delta which is 1 or 0 if $l=ip$ for some $0\leq i < d$,
or not, respectively.  We write $i\not \in v$ if $i$ labels a free
line in $v$.  Finally, note that there are $n$ nodes on the top of a
dangle and $n-2l$ on the bottom of a dangle.
%  If we wish to move a permutation through the tensor product, we
%  therefore must fix some notation.  Let $s_{i,i+1}\in B_{m,1,n}$ be
%  the element swapping the $i$ and $(i+1)$ nodes  
If the $i$th node on the top of the diagram is a free node, we let
$\underline{i}$ denote the corresponding node on the bottom of the
dangle.

From Theorem \ref{barstand}, we deduce that the action of the
generators of $\Bone$ (under our assumption on the parameter
$\delta\in k^m$ from Section \ref{deltapa}) on the standard module
$\overline{\Delta(\lambda)}$ is as follows:
 \begin{align*}
\overline{\pi}_\lambda (t_1)( v \otimes \ft  ) &= \begin{cases}
\xi^{\mathfrak{t} (1)}( v \otimes \ft  )	 &\text{ if } 1 \not\in v		\\
(t_1 v) \otimes \ft  		&\text{ if } 1=i_p \text{ for some } p
\end{cases}
\\
\overline{\pi}_\lambda(s_{i,i+1})( v \otimes \ft  ) &= 
\begin{cases}
\frac{1}{a (\underline{i},\underline{i+1})} (v \otimes \ft  ) + \left(1 + \frac{1}{a (\underline{i},\underline{i+1})}\right)\! (v \otimes \ft_{\underline{i} \leftrightarrow \underline{i+1}}  )\!\!\!\! &\text{ if } i,i+1 \not \in v		\\ 
(s_{i,i+1}v) \otimes \ft  
 &\text{ otherwise}  
 \end{cases} \\
%\overline{\pi}_\lambda(s_{1,2}^\ast)(v\otimes \ft) &=
%\begin{cases}
%\xi^{\mathfrak{t}  (1)-\mathfrak{t}  (2)} \overline{\pi}_\lambda(s_{1,2})(v\otimes \ft) &\text{ if } 1,2 \not\in v\\
%\xi^{\mathfrak{t}  (1) } \overline{\pi}_\lambda(s_{1,2})((t_2^{-1}v)\otimes \ft) &\text{ if } 1 \not\in v, 2=i_p  \\
% \overline{\pi}_\lambda(s_{1,2})((t_1t^{-1}_2v)\otimes \ft) &\text{ if }  1=i_p, 2=i_q \\
%\end{cases}
%\\
\overline{\pi}_\lambda (e_{1,2})( v \otimes \ft  ) &=\begin{cases}
0 &\text{ if }1,2\not\in v	\\
%\epsilon_{i,  j} 
\epsilon_{l,ip} \delta_{ip}  %\overline{\pi}_\lambda(s_{1,2})
( (t ^{-l}_1v) \otimes \ft) 
&\text{ if }1= i_p, 2=j_p   \\
%\epsilon_{\mathfrak{t}(1),\sgn(p)}
\xi^{l\mathfrak{t}  (1)}\overline{\pi}_\lambda(s_{1,j_p})(( t ^{-l}v) \otimes \ft)
&\text{ if }1 \not\in v, 2 =i_p  %\text{ for some } t^jp 
\\
\overline{\pi}_\lambda(s_{1,j_p})( ( t_{1}^{l} t_{2}^{-l}v )\otimes \ft)
	&\text{ if }1=i_q,  2 =i_p %\text{ for some } t^lp\neq t^kq  
\end{cases}
 \end{align*}
    The case of $\overline{\pi}_\lambda (e_{1,2})$ is symmetric in the
    coordinates $1,2$ and so we have omitted the details.

We recall that $\{t_1^p,s_1^\ast, e_{1,2} , s_{i,i+1} : 0\leq q < r, 1\leq i
\leq n-1\}$ generate $B_{m,p,n}$ for $n>2$, and that the quotient
$B_{m,1,n}/B_{m,p,n}$ is cyclic, generated by $t$.  Let $\chi$ be the
generator of the group of linear characters of the quotient which maps
$t$ to $\xi^d$.  From the formulae for the action of $t_1, s_{1,2}^\ast$,
$e_{1,2}$, and the $s_{i,i+1}$ for $1\leq i \leq n-1$, we see that
%%%%%%%%%%%%%%%%%%%%%%%%%%%%%%%%%%%%%%%%%%%%%%%%%%%%%%%%%%%%%%%%%%
the map 
$$		\sigma (v \otimes \ft ) =% (\xi_p^{p/t})^r 
\xi^{dq}	v \otimes \sigma(\ft) ,		$$
where $q$ is the total label on $v$,
induces an isomorphism
$\chi \otimes 	\overline{\pi}_\lambda=	\overline{\pi}_{\sigma(\lambda)}. 	$

 It is easy to check that $\sigma^t$ commutes with
the action of $\rho_\lambda(s_{i,i+1})$, $\rho_\lambda(s_1^\ast)$ and
$\rho_\lambda(e_{1,2})$, and that $\sigma^t \circ\rho_\lambda(t_1)=
\xi^{dt}\rho_\lambda(t_1) \circ \sigma^t$.  It follows that $\sigma^t$
commutes with the action of $B_{m,p,n}$ and that 
$$\overline{\Delta_n(\lambda)}
\!\!\downarrow_{B_{m,p,n}}^{B_{m,1,n}}\cong\oplus_{0\leq r <p/t} \ker(\sigma^t-\xi^{dtr})\overline{\Delta_n(\lambda)},$$
(although these direct summands need not be indecomposable).
% (it is yet to be proved that these are standard!). 
  For a given $0\leq r <p/t$, we have that the projection onto
  $\ker(\sigma^t-\xi^{dtr})$ is given by:
  %$$p_r=\frac{1}{t}\sum_{0\leq i < p/t}\xi^{-ir} \sigma^{id}.$$
 $$p_r=\frac{t}{p}\sum_{0\leq i < p/t}\xi^{-idtr}\sigma^{it}.$$

\begin{thm}\label{3.3.1}
The restriction of a standard module, $\overline{\Delta(\lambda)}$, for $B_{m,1,n}$ is  a direct sum of $p/t$ standard
modules for $B_{m,p,n}$.  For $\lambda \in
\Lambda(m,1,n)$, we have that 
\begin{align*}
 \Delta_n(\lambda^r) \cong \ker((\sigma^t-\xi^{dtr})\overline{\Delta_n(\lambda)}) .
\end{align*}
 \end{thm}
  
\begin{proof}
%We shall show that
%$\ker(\sigma^t-\xi^{dtr})(\overline{\Delta_n(\lambda)}) =\{ v \otimes
%x : v \in V_l^q, \ft^{q+r} \in \mathbf{S}(\lambda^{q+r}), 0 \leq q <p/t
%\}$.
It suffices to show that   $p_r \overline{\Delta(\lambda)}$ is the standard module constructed in the previous section.
For a given $x \in V_l\otimes \overline{\textbf{S}(\lambda)}$, 
% consider the subspace 
%$$\{	v \otimes \ft : v\in V_l	\}    \subseteq \overline{\Delta_n(\lambda)}$$
%%%%%
%and apply the projection $p_r$ to this subspace.
we have that 
\begin{align*}
\sigma^t (v \otimes \ft  ) 	&= \xi^{dqt} v \otimes \sigma^t(\ft)
%	\\	 
%				 	&=  v \otimes   \sigma^t( t_1^qx)
\end{align*} where $q$ is the label total on $v$, and therefore \begin{align*}
p_r (v \otimes \ft ) &=    v \otimes \left( \frac{t}{p}\sum_{0\leq i < p/t}	\xi^{-idtr +dqt}	\sigma^{it}\ft \right)
 =  v \otimes \ft^{q+r}.
	\end{align*}These elements form the basis of $\Delta(\lambda^r)$ given in Theorem \ref{basisss}, and the result follows.
  \end{proof}

\subsubsection{Simple	and projective   modules}  
By Clifford theory \cite[Theorems 1.1 and 1.3]{ramram}, we have that the simple
$B_{m,1,n}$-module, $\overline{L(\lambda)}$, restricts to a direct sum of simple $B_{m,p,n}$-modules. 
%%%%%%%%%%%%%%%
 As each simple $B_{m,1,n}$-module appears as the head of the unique standard module with the same label, we have that $$\overline{L(\lambda)}\!\!\downarrow^{B_{m,1,n}}_{B_{m,p,n}} \cong \oplus_{0 \leq r <
 p/t}L(\lambda^r),$$ by Frobenius reciprocity and Theorem \ref{3.3.1}. 
As $L(\lambda)$ is a quotient of $\overline{\Delta(\lambda)}$, we can use the action of $\sigma^t$ on the quotient to characterise $L(\lambda^r)$ as $\ker(\sigma^t-\xi^{dtr})\overline{L(\lambda)}$.

The algebra,
$B_{m,1,n}$, is free as a $B_{m,p,n}$-module.  Therefore the restriction of a 
projective  module, is   projective.    By Frobenius
reciprocity, 
%the projective $B_{m,p,n}$-module $\overline{P(\lambda)}$  restricts to
%$\oplus_{0 \leq r <
% p/t} P(\lambda^r)$.  
$$\overline{P(\lambda)}\!\!\downarrow^{B_{m,1,n}}_{B_{m,p,n}} \cong \oplus_{0 \leq r <
 p/t}P(\lambda^r).$$ 
 
 For $\lambda
\in \Lambda(m,1,n)$, the projective module $\overline{P(\lambda)}$ appears as
quotient (in fact, a direct summand) of
$$B(\lambda)=B_{m,1,n}e_{n-2l} \otimes_{B_{m,1,n-2l}} \overline{\mathbf{S}(\lambda)}.$$
 We can therefore construct the projective modules as the eigenspaces of the automorphism $\sigma^t$ (by first extending  the $\sigma^t$-action to the module $B(\lambda)$ in the obvious way).
 
 \subsubsection{Restriction to the group algebra} We now calculate the structure of the projective, standard, and simple modules for $B_{m,p,n}$ upon restriction to $kG(m,p,n)$.

\begin{prop}\label{lr2}
Let $\lambda^r, \mu^q \in \Lambda(m,p,n)$, with
$\Stab_{\ZZ/p\ZZ}(\lambda)$ $=\langle \sigma^t\rangle$ and
$\Stab_{\ZZ/p\ZZ}(\mu)$ $=\langle \sigma^u\rangle$.  We have for a
simple, projective, or standard $B_{m,p,n}$-module $M(\lambda^r)$,
that
$$[M(\lambda^r)\!\!\downarrow_{kG(m,p,n)}:\mathbf{S}(\mu^q)]=\begin{cases}
%\sum_{\rho \in T}
% {\hcf(t,u)}
 \sum_{\rho \in T}
[\overline{M(\lambda)}:\overline{\mathbf{S}(\mu^\rho )}]	& \text{if } r = q \text{ modulo } \hcf(\tfrac{p}{t}, \tfrac{p}{u})
\\
0 &\text{otherwise}
\end{cases}$$
where $T$ is a set of cosets for $\langle \sigma^{\hcf(t,u)}\rangle\leq
\ZZ/p\ZZ$.
%  \langle \sigma^{ u}\rangle$.   
%%Therefore $\Hom_{B_{m,p,n}}(M(\lambda^r),$ $M(\mu^q))=0$ if $r \neq q$
%modulo $ \hcf(p/t, p/u)$. 
\end{prop}
\begin{proof}
The multiplicities $[\overline{M(\lambda)}:\overline{\mathbf{S}(\mu)}]$ are calculated in terms of Littlewood--Richardson coefficients in \cite[Appendix]{bcd}.  From this result, it is immediate that
$$[\overline{M(\lambda)}\!\!\downarrow_{G(m,1,n)}:\overline{\mathbf{S}(\mu)}]=[\overline{M(\lambda^\sigma)}\!\!\downarrow_{G(m,1,n)}:\overline{\mathbf{S}(\mu^\sigma)}].$$

We have that $\langle \sigma^{t}\rangle$ fixes $\overline{M(\lambda)}$ and
$\langle \sigma^{u}\rangle$ fixes $\overline{\mathbf{S}(\mu)}$.  Therefore 
$$[\overline{M(\lambda)}\!\!\downarrow_{G(m,1,n)}:\mathbf{S}(\mu)]
=[\overline{M(\lambda)}\!\!\downarrow_{G(m,1,n)}:\overline{\mathbf{S}(\mu^{ {\tau}})}] $$
for $\tau \in \langle\sigma^{\hcf(t,u)}\rangle$.
  Therefore, we want to calculate the (well-defined)
multiplicities 
$$[\overline{M(\lambda)}\!\!\downarrow_{kG(m,p,n)} : (\oplus_{\tau \in \langle\sigma^{\hcf(t,u)}\rangle }\overline{\mathbf{S}(\mu^{\tau})})\!\!\downarrow_{kG(m,p,n)} ].$$   

First, note that we can factorise the map $(\sigma^t - \xi^{dtr})$ as
the product
  $$(\sigma^{ t} - \xi^{dtr}) = \prod_{0\leq i < t}(\sigma 	- \xi^{d(r+ip/t)}).$$

Now, consider the kernel of the map $(\sigma^t-\xi^{dtr})$ applied to
the direct sum. We have that 
\begin{align*}
\ker(\sigma^{ t} - \xi^{dtr})(\oplus_{\tau \in \langle\sigma^{\hcf(t,u)}\rangle }\mathbf{S}(\mu^{\tau})) &= \bigoplus_{0\leq i < t}   \ker( \sigma 	- \xi^{d(r+ip/t)})(\oplus_{\tau \in \langle\sigma^{\hcf(t,u)}\rangle }\mathbf{S}(\mu^{\tau}))  \\
&= 
\bigoplus_{
\begin{subarray}
c  r= q \text{ mod} \\
\hcf( p/t ,  {p}/{u})
\end{subarray} 
}	  \mathbf{S}(\mu^q).
 \end{align*}
Summing over a set of coset representatives of  $(\ZZ/p\ZZ)/\langle \sigma^{\hcf(t,u)}\rangle$ we obtain the desired result.
%The multiplicities $[M(\lambda^r)\!\!\downarrow_{G(m,p,n)}:\mathbf{S}(\mu^q)]$  result by summing over a set of  coset representatives for $\langle \sigma^{u}\rangle/\langle \sigma^{\hcf(t,u)}\rangle$.
 \end{proof}

 \section{Homomorphisms between standard  and projective modules}

 Let $\lambda,\mu \in \Lambda(m,1,n)$ with
 $\Stab_{\ZZ/p\ZZ}(\lambda)=\langle\sigma^{t}\rangle$ and
 $\Stab_{\ZZ/p\ZZ}(\mu)=\langle\sigma^{u}\rangle$.  Let $0 \leq r <
 p/t$ and $0 \leq q < p/u$.  We let $\lambda^r,\mu^q \in
 \Lambda(m,p,n)$ denote the elements corresponding to the $r$th and
 $q$th orbits.
 
 \begin{lem}\label{above}
Consider the algebra $B_{m,1,n}$ with parameter $\delta\in k^m$ as in
Section \ref{deltapa}.  Let $\lambda,\mu \in \Lambda(m,1,n)$  Let $\overline{M(\lambda)}$ be a standard or projective
module labelled by $\lambda$.  Let $\overline{N(\mu)}$ be a simple, standard, or
projective module labelled by $\mu$.  We have that
$$\Hom_{B_{m,1,n}}(\overline{M(\lambda)},\overline{N(\mu)}) \cong
\Hom_{B_{m,1,n}}(\overline{M(\lambda^{\sigma})},\overline{N(\mu^{\sigma})}) 
 $$
%for $\rho \in \langle \sigma \rangle$.  
%Therefore,
%$$\oplus_{\rho \in \langle\sigma\rangle/\langle\sigma^u\rangle }\Hom_{B_{m,1,n}}(\overline{M(\lambda)},\overline{N(\mu^\rho)}) \cong
%\oplus_{\rho \in \langle \sigma\rangle / \langle \sigma^u,\sigma^t \rangle}  \Hom_{B_{m,1,n}}(\overline{M(\lambda)},\overline{N(\mu^{\rho})})^{u/\hcf(t,u)} 
% $$
 \end{lem}
 \begin{proof}
The condition on the parameter implies that
 $$\overline{\delta}_{k}= \overline{\delta}_{ip+k}$$ for all $i$ and
$0\leq k\leq p-1$.  Therefore, by the un-oriented version of
\cite[Corollary 5.5.2]{bcd} outlined in the Appendix, we have that rotating both partitions by
$ip$ places results in the required isomorphism.
 \end{proof}

\begin{rmk}Note that the case that $\overline{M(\lambda)}$ is simple is excluded, as we may only use \cite[Corollary 5.5.2]{bcd} for modules $\overline{M(\lambda)}$ with a $\Delta$-filtration. 
% The second isomorphism can easily be seen to follow by removing the
% `overcount'. 
 \end{rmk}
 \begin{lem}\label{above1} Let $M(\lambda^r)$ and $N(\mu^q)$ be  simple, standard, or
projective modules labelled by $\lambda^r, \mu^q\in\Lambda(m,p,n)$.
 We have the following isomorphism:
  \begin{align*} \Hom_{B_{m,p,n}}( M(\lambda^r),N(\mu^q)) 
&\cong %\oplus_{0\leq r \leq \hcf(p/t,p/u)}  
\Hom_{B_{m,p,n}}({M(\lambda^{r+1})} ,N(\mu^{q+1})) 
\end{align*}
 \end{lem}
\begin{proof}
This follows by twisting both modules under   conjugation  by $t_1^d$.
\end{proof}
% \begin{proof}
%Consider the twisting action, $\tau$, given by conjugation by $t_1$.  Under this action we have isomorphisms:
%\begin{align*} \Hom_{B_{m,p,n}}( 	\Delta(\lambda^r),\Delta(\mu^q)) 
%&\cong \Hom_{B_{m,p,n}}({\Delta(\lambda^{r+1})} ,\Delta(\mu^{q+1})).
%\end{align*}
%Continued twisting leaves us to consider the equivalence classes under $\tau$.   Consider the subgroup $\langle \sigma^t,\sigma^u\rangle \leq \langle \sigma^t \rangle$.  There are $\hcf(p/t,p/u)$ cosets, each of size $\lcm(t,u) /t$.  %Summing over these, we are done.
%By Lemma \ref{lr2}, any module labelled by an element of a coset other than that  containing the elements congruent to $q$ modulo $\hcf(p/t,p/u)$ has no non-zero homomorphism into $\Delta(\mu^{q})$ and so we are done.
% \end{proof} 
 We let  $\epsilon_{r,q}$ denote the Kronecker delta of $r$ and $q$
 modulo $\hcf(p/t,p/u)$. 
 
 \begin{thm}\label{BH} Let $\lambda^r,\mu^q\in\Lambda(m,p,n)$.  Let $M(\lambda^r)$ be a standard or projective
module labelled by $\lambda^r$.  Let $N(\mu^q)$ be a simple, standard, or
projective module labelled by $\mu^q$. We have isomorphisms
\begin{align*}\Hom_{B_{m,p,n}}(M(\lambda^r),N(\mu^q)) &\cong \epsilon_{r,q}	\Hom_{B_{m,1,n}}(\overline{M(\lambda)},\oplus_{\rho \in (\ZZ/p\ZZ)/  		  \langle \sigma^u,\sigma^t \rangle					}
 \overline{N(\mu^\rho)} )  \\
& \cong \epsilon_{r,q}\Hom_{B_{m,1,n}}(\oplus_{\rho \in (\ZZ/p\ZZ)/  		  \langle \sigma^u,\sigma^t \rangle					}
\overline{M(\lambda^\rho)},\overline{N(\mu )} )  
.\end{align*}
  \end{thm}

\begin{proof}
 We first focus on the righthand side.  By Clifford theory, we have that  
\begin{align*}
\Hom_{B_{m,p,n}}(\oplus_{0\leq i < p/t}	M(\lambda^i),N(\mu^q)) 
&\cong\Hom_{B_{m,p,n}}(\overline{M(\lambda)}\!\!\downarrow,N(\mu^q))\\
&\cong\Hom_{B_{m,1,n}}(\overline{M(\lambda)} ,N(\mu^q)\!\!\uparrow)\\
&\cong\Hom_{B_{m,1,n}}(\overline{M(\lambda)},
\oplus_{\rho \in (\ZZ/p\ZZ)/\langle\sigma^u\rangle}\overline{N(\mu^\rho)} ), \\
&\cong
\oplus_{\rho \in (\ZZ/p\ZZ)	  / \langle\sigma^u\rangle				}	 
\Hom_{B_{m,1,n}}(\overline{M(\lambda)},
\overline{N(\mu^\rho)} ).
%%%%%%
\end{align*}
Therefore by Lemma \ref{above}, we have that
\begin{align*}
\Hom_{B_{m,p,n}}(\oplus_{0\leq i < p/t}	M(\lambda^i),N(\mu^q)) 
&\cong \!\!\! 
\bigoplus_{\begin{subarray}c \rho \in (\ZZ/p\ZZ)/  \\		  \langle \sigma^u,\sigma^t \rangle			\end{subarray}		}
\!\!\!\ 
  \Hom_{B_{m,1,n}}(\overline{M(\lambda)},  
  \overline{N(\mu^{\rho})})^{u/\hcf(t,u)} .
\end{align*}

We now focus on the lefthand side. Any $B_{m,p,n}$-homomorphism must
restrict to a $G({m,p,n})$-homomorphism, therefore 
%Therefore we have that
\begin{align*}
\Hom_{B_{m,p,n}}(\oplus_{0\leq i < p/t}	M(\lambda^i),N(\mu^q)) &\cong  
\Hom_{B_{m,p,n}}(\oplus_{
\begin{subarray}
l  r= q \text{ mod} \\
\hcf( p/t ,  {p}/{u})
\end{subarray} 
}\!\!\!	  {M(\lambda^r)} ,N(\mu^q))	%^{\lcm(t,u) /t}.  			 
\end{align*}
as all the other hom-spaces are zero, by Proposition \ref{lr2}.  By
repeated application of Lemma \ref{above1}, we get that  all the
summands on the righthand side are isomorphic, and so 
\begin{align*}
\Hom_{B_{m,p,n}}(\oplus_{0\leq i < p/t}	M(\lambda^i),N(\mu^q)) &\cong  
\Hom_{B_{m,p,n}}(   {M(\lambda^r)} ,N(\mu^q)) ^{u/\hcf(t,u)},
\end{align*}
therefore the results follows.
 \end{proof}
 
 \section{Decomposition numbers for   $B_{m,p,n}$}
 We now use Theorem \ref{BH} and Brauer--Humphrey's reciprocity to
 calculate the decomposition numbers for the Brauer algebras of type
 $G(m,p,n)$.  For $m,p,n\in\NN$, we let
 $$d_{\lambda^r,\mu^q}^{m,p,n}(\delta)=[\Delta_n(\lambda^r):L_n(\mu^q)]$$ denote
 the multiplicity of $L_n(\mu^q)$ in $\Delta_n(\lambda^r)$ as a
 $B_{m,p,n}$-module. By Brauer--Humphrey's
 reciprocity $$d_{\lambda^r,\mu^q}^{m,p,n}(\delta)=\dim_k(\Hom
 (P_n(\mu^q),\Delta_n(\lambda^r) ).$$ In \cite{marbrauer, cdv}, the
 decomposition numbers, $d_{\lambda,\mu}^{1,1,n}$, for the classical
 Brauer algebra (i.e. the type $G(1,1,n)$ case) are given by the
 corresponding parabolic Kazhdan--Lusztig polynomials of type
 $(D_n,A_{n-1})$.
 
The type $G(m,1,n)$ is covered in \cite{bcd}.  In \cite[Appendix]{bcd}
it is shown that the decomposition numbers for the un-oriented
cyclotomic Brauer algebras are as follows:
   $$d_{\lambda,\mu}^{m,1,n}(\delta)=\prod_{0\leq i < m}d_{\lambda_i,\mu_i}^{1,1,n}(\overline{\delta_i}). %\ldots d_{\lambda_m,\mu_m}^{1,1,n}(\overline{\delta_m})
   $$

By Theorem \ref{BH}, Brauer--Humphrey's reciprocity, and the above, we
have the following description of the decomposition numbers of Brauer
algebras of type $G(m,p,n)$.
\begin{thm}
The decomposition numbers, $d^{m,p,n}_{\lambda^r,\mu^q}(\delta)$ for
$B_{m,p,n}$ over a field of characteristic zero are as follows:
$$d^{m,p,n}_{\lambda^r,\mu^q}(\delta)	 = 	\epsilon_{r,q} %	\!\!\!	
\sum_{\begin{subarray}c \rho \in (\ZZ/p\ZZ)/  \\		  \langle \sigma^u,\sigma^t \rangle			\end{subarray}		}
\!\!\!\ 
%_{\begin{subarray}c \rho \in \langle\sigma\rangle   		/  \langle \sigma^{\hcf{(u,t)} }\rangle\end{subarray}	}
d^{m,1,n}_{\lambda,\mu^\rho}(\delta).
	$$
%	where the sum is over a set of coset representatives of $( \ZZ/p\ZZ)   		/  \langle \sigma^{\hcf{(u,t)} }\rangle$.
\end{thm}

\begin{rmk}
In \cite[Remark 5.5.3]{bcd} it is noted that one can reduce the calculation of certain higher extension groups for $B(m,1,n)$ to the case of the classical Brauer algebra, as we did above for the decomposition numbers.  In these cases one can calculate the corresponding higher extension groups for $B_{m,p,n}$ in a similar fashion to the above.
\end{rmk}

 \begin{Acknowledgements*}
 We would like to thank Maud De Visscher for %valuable input at the beginning of this project.
 help identifying the family of Brauer algebras studied in this paper.
 We would also like to thank Jean Michel and Shona Yu for helpful discussions
 concerning complex reflection groups,   %and the organisers of
 as well as the Centre International de Rencontres Math\'ematiques and the organisers of   `Lie theory and quantum analogues'   for   providing support and a stimulating environment during the  conference, where this project began.  The first author   is grateful for  financial support received from the ANR grant ANR-10-BLAN-0110.
 \end{Acknowledgements*}
 % 
% \begin{cor}
% For $\lambda^r,\mu^q\in\Lambda(m,p,n)$, and let $M(\lambda^r)$ and $N(\mu^q)$ be standard or projective modules labelled by $\lambda^r$, $\mu^q$.  we have isomorphisms
%\begin{align*}\Hom_{B_{m,p,n}}(M(\lambda^r),N(\mu^q)) &\cong \epsilon_{r,q}	\Hom_{B_{m,1,n}}(\overline{M(\lambda)},\oplus_
%{\rho=\rho_1\rho2, \rho_1 \in (\ZZ/p\ZZ)/  (\ZZ/p\ZZ)	, \rho_2\in	 (\ZZ/p\ZZ) \langle \sigma^u,\sigma^t \rangle}
% \overline{N(\mu^\rho)} )  
%.\end{align*}

% \end{cor}
 
%\begin{cor}
%The decomposition numbers for $B_{m,p,n}$ are given by the cyclic Kazhdan--Lusztig polynomials.  More precisely:
%$$[\Delta(\lambda^r):L(\mu^q)]=\epsilon_{r,q}
%d^p_{\lambda,\sigma}.$$
%\end{cor}
% \begin{proof}
% In the type $(1,1,n)$ case this follows from \cite{cdm}.  In type $(m,1,n)$ this follows from \cite{bcd}.  Now the type $(m,p,n)$ case follows from the above result. 
% \end{proof}
%
%
%
%
%
%

\bibliographystyle{amsalpha} \bibliography{books}

 \end{document}